\documentclass[a4paper,11pt,reqno,twoside]{amsart}
\usepackage{graphicx}
\usepackage{epsbox}
\usepackage{subfigure}
\usepackage{color}
\usepackage{cite}
\usepackage{ifpdf}
\usepackage{array}
\usepackage{slashbox}
\usepackage[T1]{fontenc}
\usepackage[latin1]{inputenc}
\usepackage{mathtools}
\usepackage{amsthm,amssymb}
\usepackage{url}

\newcommand*{\C}{\mathbb{C}}
\newcommand*{\R}{\mathbb{R}}
\newcommand*{\Z}{\mathbb{Z}}
\newcommand*{\N}{\mathbb{N}}

\renewcommand*{\geq}{\geqslant}
\renewcommand*{\leq}{\leqslant}
\newtheoremstyle{erdfn}
  {}
  {}
  {\itshape}
  {}
  {\bfseries}
  {}
  { }
  {}
\newtheoremstyle{erthm}
  {}
  {}
  {\itshape}
  {}
  {\bfseries}
  {}
  { }
  {}
\newtheoremstyle{errem}
  {}
  {}
  {}
  {}
  {\bfseries}
  {}
  { }
  {}
\theoremstyle{erthm}
\newtheorem{theorem}{Theorem}[section]

\newtheorem{lemma}[theorem]{Lemma}

\theoremstyle{erdfn}

\theoremstyle{errem}

\newtheorem{remark}{Remark}

\numberwithin{equation}{section}
\title[On monotonicity of certain weighted summatory functions associated with $L$-functions]%
      {On monotonicity of certain weighted summatory functions associated with $L$-functions} 
\author[M. Suzuki]{Masatoshi Suzuki}

\address{Department of Mathematics, Tokyo Institute of Technology, 2-12-1 Ookayama, Meguro-ku, Tokyo 152-8551, Japan}
\email{msuzuki@math.titech.ac.jp}

\subjclass[2000]{}
\AtBeginDocument{%
\mathtoolsset{showonlyrefs,mathic = true}
\maketitle
}
\begin{document}

Dedicated to Professor Akio Fujii on the occasion of his retirement from Rikkyo University. 
\bigskip

\section{Introduction}

Let $\N$ be the set of natural numbers $n=1,2,\cdots$. 
For a sequence of real numbers $\{c(n)\}_{n \in \N}$ and a real-valued locally integrable function $g:(0,\infty) \to \R$, 
we consider a weighted summatory function
\begin{equation} \label{101}
h(x) = \sum_{n=1}^{\infty} c(n) g\left(\frac{n}{x}\right). 
\end{equation}
Typical example of a weight function is the step function $g_{\rm st}$ 
which is defined by $g_{\rm st}(y)=1$ for $0<y \leq 1$, and $g_{\rm st}(y)=0$ for $y>1$. 
In this case, we obtain the usual summatory function $\sum_{n \leq x} c(n)$.  
\smallskip

As is often the case with an arithmetically defined sequence $\{c(n)\}_{n \in \N}$, 
a certain reasonable estimate $h(x) = O(x^A)$ ($x \to \infty$)
is a sufficient or an equivalent condition to the Generalized/Grand Riemann Hypothesis (GRH for short) for some zeta/$L$-function. 
On the other hand, the monotonicity of $\int_{c}^{x} h(t)\,dt$ for large $x \geq c>0$ 
may also be a sufficient or an equivalent condition to the GRH for some zeta/$L$-function. 

In the present paper we study the latter type conditions for a family 
of weighted summatory functions with certain specific weights 
in terms of the sign of $h(x)$, 
since the monotonicity of $\int_{c}^{x}h(t)\,dt$ for large $x$ 
is equivalent to the condition that $h(x)$ has a single sign for large $x$.     
\smallskip

We mention two examples of the monotonic condition. 
The first one is P\'olya's conjecture. 
Let $\lambda(n)=(-1)^{\Omega(n)}$ be the Liouville function, where $\Omega(n)$ is the number of all prime factors of $n$ 
counted with multiplicity. P\'olya~\cite{Po} conjectured that the value of the summatory function
\begin{equation} \label{102}
\sum_{n=1}^{\infty} \lambda(n)\cdot g_{\rm st}\left(\frac{n}{x}\right)
= \sum_{n \leq x} \lambda(n) 
\end{equation}
is nonpositive for $x \geq 2$ and noted that it is a sufficient condition for 
the Riemann Hypothesis (RH for short), but it was disproved by Haselgrove~\cite{MR0104638}. 
However, recently, P\'olya's approach resurrected by Ram Murty ~\cite{MR1986103} 
by considering an analogue of the P\'olya conjecture to $L$-functions of elliptic curves.  
\smallskip

The second example is Chebyshev's conjecture. 
Let $\varpi:\N \to \{0,1\}$ be the characteristic function of odd prime numbers. 
Chebyshev~\cite{Che} asserted  that the weighted summatory function 
\begin{equation} \label{103}
\sum_{n=1}^{\infty} (-1)^{\frac{n-1}{2}}\varpi(n) \cdot g\left(\frac{n}{x}\right) 
\quad (g(y)=\exp(-y))
\end{equation}
is nonpositive for large $x>0$ without a proof. 
Subsequently, Hardy-Littlewood \cite{MR1555148} and Landau \cite{MR1544293} proved independently that 
Chebyshev's assertion is equivalent to the GRH for the Dirichlet $L$-function $L(s,\chi_4)$ 
associated with the primitive non-principal Dirichlet character $\chi_4$ mod $4$. 
Moreover, Knapowski-Tur\'an \cite{MR0272729} proved that the nonpositivity of the weighted summatory function 
\begin{equation} \label{104}
\sum_{n=1}^{\infty} (-1)^{\frac{n-1}{2}}\varpi(n)\log n \cdot g\left(\frac{n}{x}\right) \quad 
(g(y)=\exp(-(\log y)^2))
\end{equation}
for large $x>0$ is also equivalent to the GRH for $L(s,\chi_4)$. 
In the later, Fujii~\cite{MR974088} generalized Chebyshev's equivalence condition to the weight function $g(y)=\exp(-y^\alpha)$ 
for every $0 < \alpha < \alpha_0$ $(\alpha_0>4)$, and conjectured that it holds for all $\alpha>0$. 
Furthermore, he proved that Knapowski-Tur\'an's equivalence condition still holds 
if their weight function is replaced by the inverse Mellin transform $2K_0(2\sqrt{y})$ of $\Gamma(s)^2$, 
where $K_n(x)$ is the $K$-Bessel function of index $n$ 
and $\Gamma(s)$ is the usual gamma function.    
\medskip

As mentioned in the final section of Fujii~\cite{MR974088}, 
we may replace the weight $g$ in \eqref{103} or \eqref{104} by a more general weight. 
Similarly, it is expected that we obtain various equivalence conditions of the RH 
by replacing the weight $g_{\rm st}$ in \eqref{102} by other reasonable weights 
as an application of several standard techniques of analytic number theory
\medskip

The main subject of the present paper is a family of specific weighted summatory functions $h_{f,\omega}^{\langle k \rangle}:(0,\infty) \to \R$ 
associated with general $L$-functions $L(f,s)$ in the sense of Iwaniec-Kowalski \cite[\S5.1]{MR2061214} 
endowed with two parameters $0<\omega<1/2$ and $k \in \N$.  
It is introduced in Section \ref{section_2}.  
The main result is that, for arbitrary fixed $k \geq 2$, 
the monotonicity of $\int_{1}^{x} h_{f,\omega}^{\langle k \rangle}(t)\,dt$ for large $x>0$ for all $0<\omega<1/2$ 
is equivalent to the GRH of $L(f,s)$. 
It is stated precisely in Section \ref{section_2} as Theorem \ref{thm_1} $\sim$ \ref{thm_4}
together with a little comment on a background of $h_{f,\omega}^{\langle k \rangle}$. 
These results are proved in Section \ref{section_4} and Section \ref{section_5} 
using two basic facts stated in Section \ref{section_3}.  
\medskip

\noindent
{\bf Acknowlegements}~
I would like to express my gratitude to Professor Akio Fujii 
for his advice and encouragement over the past several years. 
In particular, I thank for his support during my JSPS research fellowship at Rikkyo University.   

A part of this work was done during my stay at the Nottingham University in August 2010 
which was partially supported by EPSRC grant EP/E049109. 
I thank I. Fesenko for his warm hospitality during my stay. 
This work was supported by Grant-in-Aid for Young Scientists (B) (21740004). 

%
%
\section{Results} \label{section_2}
%
%
We start from two special cases of the main result (Themore \ref{thm_3} in below) for the simplicity of statements. 
That are cases of the Riemann zeta function and Dirichlet $L$-functions. 
The main result will be stated 
after these two special cases and a brief introduction of general $L$-functions. 
We shall give a little comment on the main theorem and its background after Theorem \ref{thm_4}.  
%
%
\subsection{Riemann zeta function} 
%
%
Let $\zeta(s)$ be the Riemann zeta function which is defined by the series 
$\sum_{n=1}^{\infty} n^{-s}$ for $\Re(s)>1$ 
and extended to a meromorphic function on $\C$ 
with the unique pole of residue $1$ at $s=1$. 
We denote by $\gamma(s)$ the factor $\pi^{-s/2} \Gamma(s/2)$
of the Riemann xi-function $\xi(s) = s(s-1) \pi^{-s/2} \Gamma(s/2) \zeta(s)$. 
Let $B(z;p,q)$ be the incomplete beta function defined by 
\begin{equation} \label{201}
B(z;p,q) = \int_{0}^{z} x^{p-1}(1-x)^{q-1} \, dx 
\quad (0 \leq z \leq 1, \, \Re(p)>0, \, \Re(q)>0).
\end{equation}
%
We use the notation 
\begin{equation} \label{202}
\beta(z;p,q) := B(p,q) - B(z;p,q) =  \int_{z}^{1} x^{p-1}(1-x)^{q-1} \, dx.   
\end{equation}

Let $0<\omega<1/2$. We define the real-valued function $g_\omega$ on $(0,\infty)$ by 
\begin{equation} \label{203}
\aligned
g_\omega(x) 
= \frac{4\omega}{2\omega-1} 
\frac{\pi^\omega}{\Gamma(\omega)}
\left\{ x^{\omega-1}
\, \beta\left(x^2, \frac{3-2\omega}{2},\omega \right) - \frac{2\omega+1}{4\omega} \, x^{-1/2} \, 
\beta\left(x^2, \frac{5-2\omega}{4},\omega \right)
\right\}
\endaligned
\end{equation}
for $0 < x < 1$, and $g_\omega(x) = 0$ for $x \geq 1$. 
In addition, we define 
\begin{equation} \label{204}
c_\omega(n) := n^{\omega}\sum_{d|n} \frac{\mu(d)}{d^{2\omega}}
\end{equation}
for natural numbers $n$, where $\mu(n)$ is the M\"obius function, i.e., 
$\mu(n)=0$ if $n$ is not a square free number, and $\mu(n)=(-1)^k$ if $n$ is the product of $k$ distinct primes. 
$n^\omega c_\omega(n)$ is called Jordan's totient function. 
Finally, we define the real-valued function $h_\omega$ on $(0,\infty)$ by 
\begin{equation} \label{205}
h_\omega(x) = \frac{1}{\sqrt{x}} \sum_{n=1}^{\infty} c_\omega(n) g_\omega\left(\frac{n}{x}\right). 
\end{equation}
Note that the sum on the right-hand side is finite for any $x>0$, 
since $g_\omega$ is supported on $(0,1]$ by its definition. 
Therefore $h_\omega$ is well-defined and supported on $[1,\infty)$. 
\begin{theorem} \label{thm_1}
Let $0 \leq \omega_0 < 1/2$. 
\begin{enumerate}
\item Assume that there exists $x_{\omega} \geq 1$ for every $\omega_0 < \omega < 1/2$ 
such that $h_\omega$ is nonnegative on $(x_{\omega},\infty)$. 
Then $\zeta(s) \not =0$ in the right-half plane $\Re(s) > 1/2 + \omega_0$. 
\item Assume that the RH is valid for $\zeta(s)$. 
Then there exists $x_{\omega} \geq 1$ for every $0 < \omega < 1/2$ 
such that $h_\omega$ is nonnegative on $(x_{\omega},\infty)$. 
\end{enumerate}
In particular the validity of the RH is equivalent to the statement that 
there exists $x_\omega \geq 1$ for every  $0<\omega<1/2$ 
such that $h_\omega$ is nonnegative on $(x_\omega,\infty)$.  
\end{theorem} 
\begin{remark} 
The function $g_\omega$ has only one zero $y_\omega$ in $(0,1)$ 
which tends to zero as $\omega \to 0^+$, and $g_\omega(x)>0$ on $(y_\omega,1)$. 
Moreover, $\sqrt{y}\,g_\omega(y) \to 1$ uniformly on any compact subset of $(0,1)$ as $\omega \to 0^+$. 
On the other hand, $c_\omega(n)>0$ for all $n \in \N$ because of the Euler product formula of $\sum_{n \geq 1}c_\omega(n) n^{-s}$. 
Therefore $h_\omega(x) > 0$ for all $1 < x < y_\omega^{-1}$ without any assumptions. 
Hence the proper interest is in values of $h_\omega(x)$ for $x>y_{\omega}^{-1}$. 
\end{remark}
%
%
\subsection{Dirichlet $L$-functions} 
%
%
Let $\chi$ be a real primitive Dirichlet character modulo $q$, and let 
\begin{equation*}
\delta = \delta_\chi = 
\begin{cases}
0 & \text{if $\chi(-1)=1$,} \\
1 & \text{if $\chi(-1)=-1$.}
\end{cases}
\end{equation*}
%
Let $L(s,\chi)$ be the Dirichlet $L$-function associated with $\chi$ 
which is defined by the series $\sum_{n=1}^{\infty} \chi(n) n^{-s}$ for $\Re(s)>1$ 
and extended to an entire function on $\C$. 
We denote by $\gamma(s,\chi)$ the factor $\pi^{-s/2}\Gamma((s+\delta)/2)$ 
of the completed $L$-function $\Lambda(s,\chi)=q^{s/2}\pi^{-s/2}\Gamma((s+\delta)/2)L(s,\chi)$. 
\smallskip

Let $0<\omega<1/2$. We define the function $g_{\chi,\omega}$ on $(0,\infty)$ by 
\begin{equation} \label{206}
g_{\chi,\omega}(x) 
= \frac{\pi^{\omega}}{\Gamma(\omega)} \frac{1}{\sqrt{x}} ~ \beta\left(x^2;\frac{1 + 2\delta - 2\omega}{4}, \omega \right) 
\end{equation}
for $0<x<1$, and $g_{\chi,\omega}(x)=0$ for $x\geq 1$, 
where $\beta(z;p,q)$ is the function defined in \eqref{202}. 
By using \eqref{204}, we define 
\begin{equation} \label{207}
c_{\chi,\omega}(n) := \chi(n)c_\omega(n)
\end{equation} 
for natural numbers $n$.  
Finally, we define the real-valued function $h_{\chi,\omega}$ on $(0,\infty)$ by 
\begin{equation} \label{208}
h_{\chi,\omega}(x) = q^{-\omega} \frac{1}{\sqrt{x}} \sum_{n=1}^{\infty} c_{\chi,\omega}(n)\, g_{\chi,\omega}\left(\frac{n}{x}\right).
\end{equation}
As well as $h_{\omega}$ of \eqref{205}, $h_{\omega,\chi}$ is well-defined and supported on $[1,\infty)$. 
\begin{theorem} \label{thm_2}
Let $0 \leq \omega_0 < 1/2$. 
\begin{enumerate}
\item Assume that $L(s,\chi)\not=0$ for real $s \in (1/2+\omega_0,1]$. 
Moreover assume that there exists $x_{\omega} \geq 1$ for every $\omega_0 < \omega < 1/2$ 
such that $h_{\chi,\omega}$ of \eqref{208} does not change sign on $(x_{\omega},\infty)$. 
Then $L(s,\chi) \not =0$ in the right-half plane $\Re(s) > 1/2 + \omega_0$. 
\item Assume that the GRH is valid for $L(s,\chi)$. 
Then there exists $x_{\omega} \geq 1$ for every $0 < \omega < 1/2$ 
such that $h_{\chi, \omega}$ does not change sign on $(x_{\omega},\infty)$. 
\end{enumerate}
In particular the validity of the GRH for $L(s,\chi)$ is equivalent to the statement that 
there exists $x_\omega \geq 1$ for every  $0<\omega<1/2$ 
such that $h_{\chi, \omega}$ does not change sign on $(x_\omega,\infty)$.  
\end{theorem} 

\begin{remark} 
We have $L(s,\chi) \not=0$ for real $s$ 
if the series $\theta_\chi(x)=\sum_{n \in {\Bbb Z}} n^{\delta}\chi(n) e^{-\pi n^2 x/q}$ 
attached to $L(s,\chi)$ has a single sign on $(0,\infty)$ by the formula  
$\Lambda(s,\chi)=\int_{0}^{\infty} \theta_\chi(x^2) \, x^{s+\delta-1} \, dx$ 
which is valid for all $s \in \C$. 
We may check whether $\theta_\chi(x)$ has a single sign by an elementary way 
if the modulo $q$ is small. 
\end{remark}

%
%
\subsection{General $L$-functions}
%
%
In order to state the main result, 
we specify the meaning of ``$L$-function'' in the present paper  
according to Iwaniec-Kowalski \cite[\S5.1]{MR2061214}. 
We say that $L(f, s)$ is an $L$-function with the symbol $f$ if we have the following data and conditions: 
\begin{enumerate}
\item[(L-1)] A Dirichlet series with Euler product of degree $d \geq 1$ 
\[
L(s, f) 
= \sum_{n=1}^{\infty} \frac{\lambda_f(n)}{n^s} 
= \prod_{p} \left(1 - \frac{\alpha_{f,1}(p)}{p^s}\right)^{-1} \cdots \left(1 - \frac{\alpha_{f,d}(p)}{p^s}\right)^{-1}
\]
with $\lambda_f(1)=1$, $\lambda_f(n) \in \C$, and $\alpha_{f,i}(p) \in \C$. 
We assume that the series and the Euler product converges absolutely for $\Re (s) > 1$, and 
the local parameters $\alpha_{f,i}(p)$ $(1 \leq i \leq d)$ satisfy 
$|\alpha_{f,i}(p)| < p$ for all prime numbers $p$. \\
\item[(L-2)] A gamma factor
\[
\gamma(f,s) = \pi^{-ds/2} \prod_{j=1}^{d} \Gamma\left( \frac{s+\kappa_j}{2}\right)
\]
with $\kappa_j \in \C$. 
We assume that the local parameters $\kappa_i$ ($1 \leq j \leq d$) are either real or come in conjugate pairs. 
Moreover $\Re(\kappa_j) > -1$. \\
%
\item[(L-3)] An integer $q(f) \geq 1$ such that $\alpha_{f,i}(p) \not= 0$ 
for $p \nmid q(f)$ and $1 \leq i \leq d$. \\
\item[(L-4)] The complete $L$-function defined by 
\[
\Lambda(f,s) = q(f)^{\frac{s}{2}} \gamma(f,s) L(f,s)
\]
admits analytic continuation to a meromorphic function for $s \in \C$ of order $1$, 
with at most poles at $s = 0$ and $s = 1$ with the same order $r \geq 0$. 
Moreover it satisfies the functional equation 
\[
\Lambda(f, s) = \varepsilon(f)\Lambda( \bar{f}, 1 - s),
\]
where $\bar{f}$ is an object associated with $f$ (the dual of $f$) 
for which $\lambda_{\bar{f}}(n) = \overline{\lambda_f(n)}$, 
$\gamma(\bar{f}, s) = \gamma(f, s)$, $q(\bar{f}) = q(f)$ 
and $\varepsilon(f)$ is a complex number of absolute value $1$.  
\end{enumerate}
%
%
%
\medskip

\noindent
It is said that $L(f,s)$ satisfies the Ramanujan-Petersson conjecture 
if for any $1 \leq i \leq d$ we have $|\alpha_{f,i}(p)|=1$ for all $p \nmid q(f)$ 
and $|\alpha_{f,i}(p)| \leq 1$ otherwise. This implies, in particular, $\lambda_f(n) \leq \tau_d(n) \ll_\epsilon n^\epsilon$ 
for the Dirichlet coefficients $\tau_d(n)$ of $\zeta^d(s)$. 
The Grand Riemann Hypothesis (GRH for short) refers to the statement that 
all zeros of $L(f,s)$ in the critical strip $0<\Re(s)<1$ lie on the critical line $\Re(s)=1/2$. 
\smallskip

The Riemann zeta function and Dirichlet $L$-functions 
are $L$-functions in this sense. 
Products $\zeta(s)^{k_0}L(s,\chi_1)^{k_1}\cdots L(s,\chi_l)^{k_l}$ ($k_j \in \Z_{\geq 0}$, $0 \leq j \leq l$) 
of them are also $L$-functions. 
Other typical examples of $L$-functions are $L$-functions $L(s,\phi)$  
associated with normalized Hecke eigen holomorphic cusp forms $\phi$. 
As for the theory of holomorphic cusp forms, see Chapter 14 of ~\cite{MR2061214}, for example. 
We refer these examples as $L$-functions of the symbol $f={\mathbf 1}$, $f=\chi$, and $f=\phi$, respectively.  
\medskip

Now we generalize Theorem \ref{thm_1} and \ref{thm_2} to $L$-functions. 
An $L$-function $L(f,s)$ is called self-dual if $f=\bar{f}$. 
Hereafter we assume that $L(f,s)$ is self-dual. 
Then coefficients $\lambda_f(n)$ are real for all $n$ by definition of the dual $\bar{f}$, and $\varepsilon(f) \in \{\pm 1\}$. 
The Rimeann zeta function and Dirichlet $L$-functions associated with real primitive characters are self-dual $L$-functions. 
An $L$-function attached to a normalized Hecke eigen holomorphic cusp form $\phi$ satisfying $W\phi=\phi$ for the Fricke involution $W$ 
is also self-dual $L$-function. 
\medskip

Let $0<\omega<1/2$. We define functions $g_{f,\omega,j}$, $1 \leq j \leq d$ on $(0,\infty)$ by 
\begin{equation} \label{209}
g_{f,\omega,j}(x) 
: = \frac{2 \pi^{\omega}}{\Gamma(\omega)} \, x^{-\omega+\kappa_j} (1-x^2)^{\omega-1}
\end{equation}
for $0<x<1$, and $g_{f,\omega,j}(x)=0$ for $x \geq 1$. In addition, we define
\[
p_1(x) : =\delta_1(x) - 2\omega(1 - 2 \omega) \, x^{\omega-1} - 2\omega(1 + 2 \omega) \, x^{\omega}
\]
for $0<x<1$, and $p_1(x)=0$ for $x>1$, 
where $\delta_1(x)$ is the delta function at $x=1$, and 
\[
p_r(x) : = \int_{0}^{1} \cdots \int_{0}^{1} p_1\left(\frac{x}{y_1 \cdots y_{r-1}}\right) \, 
p_1(y_1) \cdots p_1(y_{r-1})\,\frac{dy_1}{y_1} \cdots \frac{dy_{r-1}}{y_{r-1}} \quad (r \geq 2). 
\] 
By using these functions, we define
\begin{equation} \label{210}
g_{f,\omega}^{\langle 0 \rangle}(x) :=  (p_r \ast g_{f,\omega,1} \ast \cdots \ast g_{f,\omega,d})(x),
\end{equation}
where $r \geq 0$ is the order of the pole of $\Lambda(f,s)$ at $s=1$  
and $\ast$ means the multiplicative convolution $(F \ast G)(x) = \int_{0}^{\infty} F(x/y)G(y) \, y^{-1}dy$. 
Note that $g_{f,\omega}^{\langle 0 \rangle}(x)=0$ for $x>1$ by its definition. 
We define numbers $\mu_f(n)$ by the Dirichlet coefficients of $L(f,s)^{-1}$: 
\begin{equation} \label{211}
\frac{1}{L(f,s)} 
= \sum_{n=1}^{\infty} \frac{\mu_f(n)}{n^s}
= \prod_{p} \left(1 - \frac{\alpha_{f,1}(p)}{p^s}\right) \cdots \left(1 - \frac{\alpha_{f,d}(p)}{p^s}\right),
\end{equation}
and define the numbers $c_{f,\omega}(n)$ by  
\begin{equation} \label{212}
c_{f,\omega}(n) = n^\omega \sum_{d|n} \frac{\mu_f(d)\lambda_f(n/d) }{d^{2\omega}}.
\end{equation} 
Moreover we define the function $h_{f,\omega}^{\langle 0 \rangle}$ on $(0,\infty)$ by 
\begin{equation} \label{213}
h_{f,\omega}^{\langle 0 \rangle}(x) = q(f)^{-\omega} \frac{1}{\sqrt{x}} 
\sum_{n=1}^{\infty} c_{f,\omega}(n) g_{f,\omega}^{\langle 0 \rangle}\left(\frac{n}{x}\right).
\end{equation}
As well as \eqref{205} and \eqref{208}, the right-hand side of \eqref{213} is a finite sum for any fixed $x \geq 1$ 
and vanishes for $0<x<1$. 
Finally we define the function $h_{f,\omega}^{\langle k \rangle}$ 
for $k \in \N$ by  
\begin{equation} \label{214}
h_{f,\omega}^{\langle k \rangle}(x):=\int_{1}^{x} h_{f,\omega}^{\langle k-1 \rangle}(y) \frac{dy}{y}.
\end{equation}
As proved in Lemma \ref{lem_401} below, $h_{f,\omega}^{\langle k \rangle}$ is a well-defined continuous function on $(0,\infty)$. 
We have 
\begin{equation} \label{217}
h_{f,\omega}^{\langle k \rangle}(x) = q(f)^{-\omega} \frac{1}{\sqrt{x}} 
\sum_{n=1}^{\infty} c_{f,\omega}(n) g_{f,\omega}^{\langle k \rangle}\left(\frac{n}{x}\right)
\end{equation}
if we put 
\begin{equation} \label{216}
g_{f,\omega}^{\langle k \rangle}(x):=\int_{x}^{1} \sqrt{\frac{y}{x}} \, g_{f,\omega}^{\langle k-1 \rangle}(y) \frac{dy}{y}.
\end{equation}
As proved in Lemma \ref{lem_401} below, $g_{f,\omega}^{\langle k \rangle}$ is a well-defined continuous function on $(0,\infty)$. 
 
The function $h_{f,\omega}^{\langle 1 \rangle}$ is equal to \eqref{205} (resp. \eqref{208}) 
if $f={\mathbf 1}$ (resp. $f=\chi$) by Lemma \ref{lem_402} below. 
Now Theorems \ref{thm_1} and \ref{thm_2} are generalized as follows: 
\begin{theorem} \label{thm_3}
Let $0 \leq \omega_0 < 1/2$. Let $L(f,s)$ be a self-dual $L$-function in the sense of the above. 
\begin{enumerate}
\item[(1)] 
Assume that $L(f,s)\not=0$ for any real $s \in (1/2+\omega_0,1]$. 
Moreover assume that the following condition holds for some natural number $k \geq 1$: 
there exists $x_{\omega,k} \geq 1$ for every $\omega_0 < \omega < 1/2$ 
such that $h_{f,\omega}^{\langle k \rangle}$ of \eqref{214} does not change sign on $(x_{\omega,k},\infty)$. 
Then $L(f,s) \not =0$ in the right-half plane $\Re(s) > 1/2 + \omega_0$. 
\item[(2-a)] Assume that the GRH is valid for $L(f,s)$. 
Then there exists $x_{\omega,k} \geq 1$ for every natural number $k \geq 2$ and every real number $0 < \omega < 1/2$ 
such that $h_{f, \omega}^{\langle k \rangle}$ does not change sign on $(x_{\omega,k},\infty)$. 
\item[(2-b)] Assume that $d=1$ $($see {\rm (L-1))} and that   
the Ramanujan-Petersson conjecture and the GRH are valid for $L(f,s)$. 
Then there exists $x_{\omega,k} \geq 1$ for every natural number $k \geq 1$ and every real number $0 < \omega < 1/2$ 
such that $h_{f, \omega}^{\langle k \rangle}$ does not change sign on $(x_{\omega,k},\infty)$. 
\end{enumerate}
\end{theorem} 
\begin{remark} 
Let $L(s,\phi)$ be a self-dual $L$-function attached 
to a normalized Hecke eigen holomorphic cusp form $\phi$ of weight $k$ and level $q$. 
If $\phi(iy)$ has a single sign on $(0,\infty)$, e.g. the Ramanujan delta function,  
$\Delta(z)=e^{2\pi i z} \prod_{n=1}^{\infty} (1-e^{2\pi i n z})^{24}$, 
we have $L(s,\phi)\not=0$ for $s \in \R$ by the integral formula 
$\Lambda(s,f) = \int_{0}^{\infty} y^{(k-1)/2} \phi( iy q^{-1/2} ) \, y^{s-1} \, dy$.   
\end{remark}
By Kaczorowski-Perelli~\cite{MR1710182}, the assertion of (2-b) for $k=1$ is essentially (2) of Theorem \ref{thm_1} and \ref{thm_2}. 
One of obvious advantage of cases of $d=1$; the Riemann zeta function ($f={\mathbf 1}$) and Dirichlet $L$-functions ($f=\chi$), is that 
we can define functions $h_{f,\omega}^{\langle 1 \rangle}$ by elementally ways only   
because of the simplicity of the coefficients $\lambda_{f}(n)$ and the gamma factor $\gamma(f,s)$. 
\smallskip

It is not obvious whether the condition $k \geq 2$ in (2-a) is relaxed to $k \geq 1$ for any $d>1$ 
by a technical reason. 
However, we obtain the following result at least.  
\begin{theorem} \label{thm_4} 
Let $L(f,s)$ be a self-dual $L$-function. 
Assume that the GRH is valid for $L(f,s)$. 
Then the function 
\begin{equation} \label{215}
R_{f,\omega}(x) := h_{f,\omega}^{\langle 1 \rangle}(x) - \varepsilon(f)\cdot{\mathbf 1}_{[1,\infty)}(x) \
\end{equation}
belongs to $L^2((1,\infty), x^{-1} dx)$ for every $0< \omega < 1/2$, 
where $\varepsilon(f)$ is the sign of the functional equation in {\rm (L-4)} and ${\mathbf 1}_{[1,\infty)}$ is the characteristic function of $[1,\infty)$. 
In other words, for large $x \geq 1$, $h_{f,\omega}^{\langle 1 \rangle}$ has the definite sign $\varepsilon(f)$  in the sense of $L^2$. 

Conversely, if \eqref{215} belongs to $L^2((1,\infty),x^{-1}dx)$ for every $\omega_0 < \omega < 1/2$, 
Then $L(f,s) \not =0$ in the right-half plane $\Re(s) > 1/2 + \omega_0$. 
\end{theorem} 

In Theorem \ref{thm_3}, we specified the special type function $g_{f,\omega}^{\langle k \rangle}$ of \eqref{216} as a weight in \eqref{101}. 
As found in \eqref{203} or \eqref{206}, 
the weight $g_{f,\omega}^{\langle k \rangle}$ has quite different form  
comparing with any other weights in the introduction 
and other usual weights studied in analytic number theory 
(see \cite[\S5.1]{MR2378655}, for example). 
One may consider that it is possible to extend Theorem \ref{thm_3} 
to more wide class of weights.  
Moreover one may wonder why we introduced the parameter $\omega$ (and $k$) which complicate the statements of results. 

However we have a positive reason for the restriction of weights to the special $g_{f,\omega}^{\langle k \rangle}$. 
It is in the connection between the weighted summatory functions $h_{f,\omega}^{\langle k \rangle}$ 
and the theory of model subspaces of the Hardy space on the upper half plane 
generated by meromorphic inner functions. In this connection the case of $k=1$ has a particular importance. 

A model subspace is the orthogonal complement of some invariant subspace of the Hardy space. 
The theory of model subspaces is one of fruitful area of function analysis and operator theory 
(see \cite{MR1864396, MR1892647}, for example). 
The existence of such background is a remarkable advantage of our summatory functions with the special weights. 
In general, it is hard to prove the monotonicity of $\int h$  
if it is sufficient or equivalent to the GRH for some zeta/$L$-function. 
A reason of such difficulty may be in a situation that 
we can not reduce the monotonicity of $\int h$ to other plausible problem inside/outside number theory, of course, 
except for an essential difficulty of the GRH itself.  
For example, as far as the author know, 
the monotonicity of primitive functions of \eqref{102}, \eqref{103}, and \eqref{104} 
are not reduced to other reasonable problems inside/outside number theory. 
See the forthcoming paper \cite{Su} for bridges between weighted summatory functions $h_{f,\omega}^{\langle 1 \rangle}$ 
and the theory of model subspaces. 
%
%
\section{Preliminaries} \label{section_3}
%
%
\noindent
{\bf Notation.} We denote by $s=\sigma+it$ the complex variable 
with the real part $\sigma$ and the imaginary part $t$, 
and use $\epsilon$ to express arbitrary small positive real number.  
For a positive valued function $g(x)$, 
we use Landau's $f(x)=O(g(x))$ and Vinogradov's $f(x) \ll g(x)$ as the same meaning in the usual sense. 
Also we use Landau's $f(x)=o(g(x))$ for $x \to \infty$ 
in the meaning that for any $\epsilon$ there exists 
$x_\epsilon>0$ such that $|f(x)| \leq \epsilon g(x)$ for $x \geq x_\epsilon$.  
\medskip

The following lemmas are used repeatedly in the later sections.

\begin{lemma} {\rm (Stirling's formula \cite[A.4 of \S5]{MR2061214})} Let $-\infty<\sigma_1<\sigma_2<\infty$. We have
\begin{equation} \label{301}  
\Gamma(\sigma+it) = \sqrt{2\pi} \, |t|^{\sigma+it-1/2}e^{-(\pi/2)|t|-it + {\rm sgn}(t) i(\pi/2)(\sigma-1/2)} (1+O(|t|^{-1}))
\end{equation}
for $\sigma_1 \leq \sigma \leq \sigma_2$ and $|t| \geq 1$. 
\end{lemma}
%
%
\begin{lemma} {\rm (\!\!\cite[(5.35) of p.195]{MR0352890})}
For $\Re(s+\alpha)>0$ and $\Re(\beta-\alpha)>0$, we have 
\begin{equation} \label{302}
\frac{\Gamma(s+\alpha)}{\Gamma(s+\beta)} 
= \frac{1}{\Gamma(\beta-\alpha)} \int_{0}^{1}  x^\alpha (1-x)^{\beta-\alpha-1}\, x^s \, \frac{dx}{x}. 
\end{equation}
\end{lemma}
%
%
\section{On the first half of Theorem \ref{thm_3}} \label{section_4}
%
%
In this section, we prove Theorem \ref{thm_3} (1). We define the entire function $\xi(f,s)$ by 
\[
\xi(f,s) := s^{r}(s-1)^r \Lambda(f,s),
\]
where $r$ is the order of the pole of $\Lambda(f,s)$ at $s=1$ (see (L-4)), and put 
\begin{equation} \label{401}
\Theta_{f,\omega}(s) := \frac{\xi(f,s-\omega)}{\xi(f,s+\omega)} 
= \left[\frac{(s-\omega)(s-\omega-1)}{(s+\omega)(s+\omega-1)}\right]^r \frac{\Lambda(f,s-\omega)}{\Lambda(f,s+\omega)}. 
\end{equation}
\begin{lemma} \label{lem_401} Let $0<\omega<1/2$ and $k \in \N$. 
Functions $g_{f,\omega}^{\langle k \rangle}$ of \eqref{216} and $h_{f,\omega}^{\langle k \rangle}$ of \eqref{214} 
are continuous functions on $(0,\infty)$ supported on $(0,1]$ and $[1,\infty)$, respectively.  
\end{lemma}
\begin{proof} By \eqref{217}, 
it is sufficient to prove that $g_{f,\omega}^{\langle k \rangle}$ is a continuous 
function on $(0,1]$ and $\displaystyle{\lim_{x\to1^-} g_{f,\omega}^{\langle k \rangle}(x)=0}$. 
By \eqref{216}, it is reduced to the case $k=1$, and we have 
\begin{equation} \label{402}
\sqrt{x} \, g_{f,\omega}^{\langle 1 \rangle}(x)=\int_{x}^{1} \sqrt{y} \,
g_{f,\omega}^{\langle 0 \rangle}(y) \, \frac{dy}{y}.
\end{equation}
Because $L^1((0,1),x^{-1}dx)$ is closed under the multiplicative convolution, 
$g_{f,\omega}^{\langle 0 \rangle}$ belongs to $L^1((0,1),x^{-1}dx)$ by definition \eqref{210} and the assumption $0<\omega<1/2$. 
Hence the right-hand side of \eqref{402} is a continuous function on $(0,1)$ and tends to zero as $x \to 1^-$. 
\end{proof}
\begin{lemma} Let $0<\omega<1/2$ and $k \in \N$. 
We have 
\begin{equation} \label{403}
\int_{1}^{\infty} h_{f,\omega}^{\langle k \rangle}(x) \, x^{\frac{1}{2}-s} \, \frac{dx}{x} = \frac{\Theta_{f,\omega}(s)}{(s-1/2)^k}   
\end{equation}
together with the absolute convergence of the integral for sufficiently large $\Re(s)>0$. 
Under the Ramanujan-Petersson conjecture, the region of the absolute convergence is relaxed to 
the right-half plane $\Re(s)>1+\omega$.  
\end{lemma}
\begin{proof}
Put $\gamma_{j}(f,s): = \pi^{-s/2} \Gamma((s+\kappa_j)/2)$ for $1 \leq j \leq d$ 
such that $\gamma(f,s) = \prod_{j=1}^{d}\gamma_j(f,s)$. 
Applying \eqref{302} to functions $g_{f,j,\omega}$ of \eqref{209}, 
we have 
\begin{equation} \label{404}
\int_{0}^{1}  g_{f,j,\omega}(x) \, x^{s} \, \frac{dx}{x}  
= \frac{\gamma_j(f,s-\omega)}{\gamma_j(f,s+\omega)} 
\quad \text{for} \quad \Re(s)>1+\omega,   
\end{equation}
since $\Re(\kappa_j) > -1$ (see (L-2)). On the other hand, 
%
%
we have 
\begin{equation} \label{405}
\int_{0}^{1} p_1(x) \, x^s \frac{dx}{x} = \frac{(s-\omega)(s-\omega-1)}{(s+\omega)(s+\omega-1)} 
\quad \text{for} \quad \Re(s)>1+\omega
\end{equation}
by elementary ways. 
Applying Theorem 44 of \cite{MR942661} to \eqref{404} and \eqref{405} with definition \eqref{210}, 
we obtain 
\begin{equation} \label{406}
\int_{1}^{\infty} g_{f,\omega}^{\langle 0 \rangle}(x) \, x^s \, \frac{dx}{x} 
= \left[ \frac{(s-\omega)(s-\omega-1)}{(s+\omega)(s+\omega-1)} \right]^r \frac{\gamma(f,s-\omega)}{\gamma(f,s+\omega)} 
\quad \text{for} \quad \Re(s)>1+\omega.
\end{equation} 
Using \eqref{216} repeatedly, we have 
\begin{equation} \label{407}
g_{f,\omega}^{\langle k \rangle}(x) = \frac{(-1)^{k-1}}{(k-1)!}\int_{x}^{1} \sqrt{\frac{y}{x}} \left(\log \frac{x}{y}\right)^{k-1} 
g_{f,\omega}^{\langle 0 \rangle}(y) \, \frac{dy}{y}, 
\end{equation}
while  
\begin{equation} \label{408}
\frac{(-1)^k}{(k-1)!} \int_{0}^{1} x^{-1/2} (\log x)^{k-1} \, x^{s} \frac{dx}{x} = \frac{1}{(s-1/2)^k} 
\quad \text{for} \quad \Re(s)>1/2.  
\end{equation}
Therefore, by applying Theorem 44 of \cite{MR942661} again to \eqref{406} and \eqref{408} with \eqref{407}, we obtain 
\begin{equation} \label{409}
\int_{1}^{\infty} g_{f,\omega}^{\langle k \rangle}(x) \, x^s \, \frac{dx}{x} 
= \frac{1}{(s-1/2)^k}\left[ \frac{(s-\omega)(s-\omega-1)}{(s+\omega)(s+\omega-1)} \right]^r \frac{\gamma(f,s-\omega)}{\gamma(f,s+\omega)} 
\quad \text{for} \quad \Re(s)>1+\omega.
\end{equation} 
By definition \eqref{211} and \eqref{212}, we have 
\begin{equation} \label{410}
\sum_{n=1}^{\infty} \frac{c_{f,\omega}(n)}{n^s} 
= \sum_{m=1}^{\infty} \frac{\lambda_f(m)}{m^{s-\omega}} \sum_{n=1}^{\infty} \frac{\mu_f(n)}{n^{s+\omega}}
= \frac{L(f,s-\omega)}{L(f,s+\omega)} 
\end{equation}
as an equality of formal Dirichlet series. 
However all Dirichlet series in \eqref{410} converge absolutely 
if $\Re(s)>0$ is sufficiently large, since we assumed $|\alpha_{f,i}(p)|<p$ (see (L-1)). 

Under the Ramanujan-Petersson conjecture, the region of the absolute convergence of the Dirichlet series \eqref{410} 
is relaxed to $\Re(s)>1+\omega$, 
since $\lambda_f(n) \ll_\epsilon n^{\epsilon}$ and $\mu_f(n) \ll_\epsilon n^{\epsilon}$. 
Hence, \eqref{217}, \eqref{409}, \eqref{410}, and Fubini's theorem derive the assertion we needed.   
\end{proof}
\begin{lemma} \label{lem_402}
The function $g_{\omega}$ of \eqref{203} $($resp. $g_{\chi,\omega}$ of \eqref{206}$)$ 
is equal to $g_{f,\omega}^{\langle 1 \rangle}$ of \eqref{214} for $f={\mathbf 1}$ $($resp. $f=\chi)$.  
\end{lemma}
\begin{proof}
This is a simple consequence of \eqref{302} and \cite[Theorem 44]{MR942661} by \eqref{201} and \eqref{202}. 
\end{proof}

\noindent
{\bf Proof of Theorem \ref{thm_3} (1).} 
We prove that $\Lambda(f,s)\not=0$ for any $s \in \C$ in the strip $1/2+\omega_0 < \Re(s) \leq 1$ by contradiction. 
Note that $\Lambda(f,s) \not= 0$ for $\Re(s)>1$ by the Euler product of $L(f,s)$ with its convergence condition 
and the assumption on the gamma factor $\gamma(f,s)$.  
By the assumptions of Theorem \ref{thm_3} (1), 
a well-known theorem of Landau (see \cite[\S5 of Chap.II]{MR0005923}, for example) and formula \eqref{403} imply that 
$\Theta_{f,\omega}(s)$ 
has no poles in the right-half plane $\Re(s)>1/2$ for every $\omega_0<\omega<1/2$.   
\medskip

Suppose that there exists a zero $\rho$ of $\Lambda(f,s)$ such that $1/2 + \omega_0 < \Re(\rho) \leq 1$ and 
$|\Im(\rho)|>0$. 
We take some $T > |\Im(\rho)|>0$. Then the set
\[
{\mathcal Z}_{T}:=\left\{
s \in \C ~\left|~ 
\Lambda(f,s)=0,~\frac{1}{2}+\omega_0 <\Re(s) \leq 1,~|\Im(\rho)|<T
\right.
\right\}
\]
is non-empty and finite. 
Therefore we may assume that $\Re(\rho)$ is minimal in ${\mathcal Z}_T$ 
by replacing $\rho$ by another zero in ${\mathcal Z}_T$ if necessary.  
\medskip

Obviously it is possible to take $(0\leq) ~\omega_0 < \omega < 1/2$ such that 
\begin{equation} \label{411}
\frac{1}{2} + \omega_0 < \frac{1}{2} + \omega < \Re(\rho) < \frac{1}{2} + \omega + (\omega - \omega_0). 
\end{equation}
For such $\omega$, we put $\tilde{s}:=\rho-\omega$. 
Then $\Lambda(f,\tilde{s}+\omega)=\Lambda(f,\rho)=0$ with 
\begin{equation} \label{412}
\frac{1}{2}< \Re(\tilde{s}) < \frac{1}{2}+\omega - \omega_0.
\end{equation}
This implies $\Lambda(f,\tilde{s}-\omega)=0$, 
since $|\Im(\rho)|>0$ and $\Theta_{f,\omega}(s)$ has no poles on $\Re(s)>1/2$. 
By the functional equation,   
$\Lambda(f,1-\tilde{s}+\omega) = \varepsilon(f) \Lambda(f,\tilde{s}-\omega)=0$. 
Hence $1-\tilde{s}+\omega$ is a zero of $\Lambda(f,s)$ having a nonzero imaginary part. 
For this zero, \eqref{412} implies 
\[
\frac{1}{2} + \omega_0 < \Re(1 - \tilde{s} +\omega) < \frac{1}{2} + \omega \,(< 1).
\]
This contradicts the choice of $\rho$ by \eqref{411}. \hfill $\Box$
%
%
%
\section{On the latter half of Theorem \ref{thm_3} and Theorem \ref{thm_4}} \label{section_5}
%
%
In this section, we prove Theorem \ref{thm_3} (2-a), (2-b), 
and Theorem \ref{thm_4}.  
We start from the preparation of lemmas. 

\begin{lemma} \label{lem_501} 
Let $k \in \N$. 
Suppose that $c_{f,\omega}(n)=O(\psi_{f,\omega}(n))$ for some positive valued arithmetic function $\psi_{f,\omega}$. 
If $c>0$ and $T>0$ are sufficiently large, we have 
\begin{equation*}
\aligned
h_{f,\omega}^{\langle k \rangle}(x) 
& =  \frac{1}{2\pi i}\int_{c-iT}^{c+iT} \frac{\Theta_{f,\omega}(s)}{(s-1/2)^k} \, x^{s-\frac{1}{2}} \, ds 
 + O\left(\frac{x^{c-\frac{1}{2}}}{T^{k+d\omega}(c-1-\omega)^r}\right) \\
& \quad  + O\left(\frac{\psi_{f,\omega}(2x)\, x^{1/2}\log x}{T^{k+d\omega}} \right)
 + O\left(\frac{\psi_{f,\omega}(x)}{T^{k+d\omega}\sqrt{x}}\right)
\endaligned
\end{equation*}
for $x>1$ with $x \not\in \Z$.  
\end{lemma}

\begin{proof}
By the Stirling formula \eqref{301}, 
in any vertical strip of finite width,  
there exists $T_0 \geq 1$ such that 
\begin{equation} \label{501}
\frac{\gamma(f,s-\omega)}{\gamma(f,s+\omega)} 
= (2\pi)^{d\omega}|t|^{-d\omega} e^{-\frac{\pi i d\omega}{2} {\rm sgn}(t)}(1+O(|t|^{-1})) \quad \text{for} \quad |t| \geq T_0.
\end{equation}
By \eqref{407} and the summability of $|g_{f,\omega}^{\langle 0 \rangle}(x)|$ on $(\epsilon,1)$, 
$g_{f,\omega}^{\langle k \rangle}$ belongs to $C^1(0,1)$ and is of locally bounded variation. 
Therefore the Mellin inversion formula of \eqref{409} holds for $0<x<1$ and $c>1+\omega$ (\!\!\cite[Theorem 28]{MR942661}):
\[
g_{f,\omega}^{\langle k \rangle}(x) = \frac{1}{2\pi i}\int_{c-i\infty}^{c+i\infty} \frac{1}{(s-1/2)^k}
\left[ \frac{(s-\omega)(s-\omega-1)}{(s+\omega)(s+\omega-1)} \right]^r
\frac{\gamma(f,s-\omega)}{\gamma(f,s+\omega)} 
\, x^{-s} \, ds.
\]
Note that we may understand as $T \geq T_0$, since we assumed that $T$ is sufficiently large. 
By \eqref{501}, we have 
\[
\aligned
\frac{1}{2 \pi i} & \int_{c+iT}^{c+i\infty} \frac{1}{(s-1/2)^k}\left[ \frac{(s-\omega)(s-\omega-1)}{(s+\omega)(s+\omega-1)} \right]^r
\frac{\gamma(f,s-\omega)}{\gamma(f,s+\omega)} 
\, x^{-s} \, ds \\
& = x^{-c} (2\pi)^{d\omega-1}e^{-\frac{\pi i d\omega}{2}} \int_{T}^{\infty} \frac{t^{-d\omega}}{(c-1/2+it)^k}
 (1+O(|t|^{-1})) \, x^{-it} \, dt \\ 
& = x^{-c} (2\pi)^{d\omega-1}e^{-\frac{\pi i d\omega}{2}} \int_{T}^{\infty} \frac{i^k \, t^{k-d\omega}}{((c-1/2)^2+t^2)^k} \, x^{-it} \, dt
+ O\left( x^{-c}T^{-k-d\omega} \right).
\endaligned
\]
The function $t^{k-d\omega}/((c-1/2)^2+t^2)^k$ in the integral of the right-hand side decreases monotonically    
if $|t|>T_1$ for some $T_1>0$. Therefore 
\[
\aligned
\left|
\int_{T}^{\infty} \frac{t^{k-d\omega}}{((c-1/2)^2+t^2)^k} \, x^{-it} \, dt 
\right| \leq  \frac{4}{\log x}\frac{T^{k-d\omega}}{((c-1/2)^2+T^2)^k}
\endaligned
\]
by the first derivative test (see \cite[\S2.1]{MR1994094}, for example) if $T \geq {\rm max}\{T_0,T_1\}$. Hence 
\[
\aligned
g_{f,\omega}^{\langle k \rangle}(x) & = \frac{1}{2\pi i}\int_{c-iT}^{c+iT} \frac{1}{(s-1/2)^k}
\left[ \frac{(s-\omega)(s-\omega-1)}{(s+\omega)(s+\omega-1)} \right]^r
\frac{\gamma(f,s-\omega)}{\gamma(f,s+\omega)} \, x^{-s} \, ds \\ 
& \quad + O\left(\frac{1}{x^c}\frac{1}{T^{k+d\omega}}\right) + O\left(\frac{1}{x^c\log x}\frac{1}{T^{k+d\omega}}\right). 
\endaligned
\]
By \eqref{213}, \eqref{401}, and \eqref{410}, we have 
\[
\aligned
h_{f,\omega}^{\langle k \rangle}(x) & = \frac{1}{2\pi i}\int_{c-iT}^{c+iT} \frac{\Theta_{f,\omega}(s)}{(s-1/2)^k} x^{s-\frac{1}{2}} \, ds  \\ 
& \quad + O\left(\frac{x^{c-1/2}}{T^{k+d\omega}} \sum_{n=1}^{\infty} \frac{|\psi_{f,\omega}(n)|}{n^c} \right) 
+ O\left(\frac{x^{c-1/2}}{T^{k+d\omega}} \sum_{n=1}^{\infty} \frac{|\psi_{f,\omega}(n)|}{n^c|\log (n/x)|}\right)  
\endaligned
\]
since $c$ is large. 
Sums in the error terms are estimated by a standard way (see \cite[\S3.12]{MR882550}, for example), 
and then we obtain the desired formula. 
\end{proof}

\begin{lemma} \label{lem_502} 
Let $0<\omega<1/2$. 
Assume that the GRH of $L(f,s)$ is valid. 
Then we have $|\Theta_{f,\omega}(s)|<1$ for $\Re(s) > 1/2$, 
and $|\Theta_{f,\omega}(s)|=1$ for $\Re(s)=1/2$. 
\end{lemma}
\begin{proof}
Recall that $\varepsilon(f) \in \{\pm 1\}$ by the self-duality of $L(f,s)$. 
Applying Theorem 4 of \cite{MR2220265} to $\xi(f,s)$, 
we obtain $|\Theta_{f,\omega}(s)|<1$ for $\Re(s) > 1/2$. 
Using the functional equation $\xi(f,s)= \varepsilon(f) \xi(f,1-s)$ in \eqref{401}, 
we obtain $|\Theta_{f,\omega(s)}|=1$ on $\Re(s)=1/2$.
\end{proof}

\begin{lemma} \label{lem_503} 
Assume that the Ramanujan-Petersson conjecture and the GRH for $L(f,s)$. 
For any $\epsilon>0$ we have 
\begin{equation}
L(f,\sigma+it) \ll  
\begin{cases}
|t|^{d \epsilon} & \text{if} \quad \sigma \geq 1/2,~|t| \to \infty, \\[5pt]
|t|^{d (\frac{1}{2}-\sigma+\epsilon)} & \text{if} \quad \sigma <1/2,~|t| \to \infty,  
\end{cases}
\end{equation}
where the implied constant depends on $f$ and $\epsilon$. 
We can take $\epsilon=0$ if $\sigma>1$ or $\sigma<0$. 
Moreover 
\begin{equation}
\frac{1}{L(f,\sigma + it)} \ll |t|^{d\epsilon} \quad (|t| \to \infty)
\end{equation}
in the right-half plane $\sigma \geq 1/2 + \epsilon$. 
\end{lemma}
\begin{proof} 
For $0 \leq \sigma \leq 1$, the estimate for $L(f,s)$ is a consequence of 
Corollary  5.20 of \cite{MR2061214} and the Phragmen-Lindel\"of convexity principle. 
We have $L(f,\sigma+it) \ll 1$ for $\sigma>1$ by the absolute convergence of the Dirichlet series, 
and $L(f,\sigma+it) \ll |t|^{d(1/2-\sigma)}$ for $\sigma<0$ 
by the functional equation and the Stirling formula \eqref{301}.  
The estimate for $L(f,s)^{-1}$ is a consequence of Theorem 5.19 of \cite{MR2061214}. 
\end{proof}

\noindent
{\bf Proof of Theorem \ref{thm_3} (2-a).}
Let $0<\omega<1/2$, $0<\delta<\omega$, and $T \geq  {\rm max}\{T_0,  T_1\}$, 
where $T_0$ and $T_1$ are positive real numbers appeared in the proof of Lemma \ref{lem_501}. 

We consider the positively oriented closed path consisting of 
the vertical line from $c - iT$ to $c + iT$,
the horizontal line from $c + iT$ to $1/2 + iT$, 
the vertical line from $1/2+iT$ to $1/2+i\delta$, 
the counter-clockwise left-half circle ${\mathcal C}_\delta$ of radius $\delta$ around $s=1/2$, 
the vertical line from $1/2-i\delta$ to $1/2-iT$, 
and the horizontal line from $1/2 - iT$ to $c - iT$. 
In the interior of the closed path 
$\Theta_{f,\omega}(s)/(s-1/2)^k$ has no poles except for the pole of order $k$ ($\geq 2$) at $s = 1/2$, 
since $\Theta_{f,\omega}(s)$ has no poles in the right-half plane $\Re(s)>1/2-\omega$ by the GRH for $L(f,s)$ 
and $\Theta_{f,\omega}(1/2) = \varepsilon(f) \in \{\pm 1\}$ by the functional equation of $\Lambda(f,s)$. 
Thus the residue theorem gives
\begin{equation} \label{temp_1}
\aligned
\frac{1}{2\pi i} & \int_{c-iT}^{c+iT} \frac{\Theta_{f,\omega}(s)}{(s-1/2)^k} \, 
x^{s-\frac{1}{2}} \, ds = P_k(\log x) \\
& + \frac{1}{2\pi i} 
\left( 
-  \int_{{\mathcal C}_\delta}
+ \int_{1/2+\delta}^{1/2+iT}
+ \int_{1/2-iT}^{1/2-i\delta}
+ \int_{1/2 +iT}^{c +iT}
- \int_{1/2 -iT}^{c -iT} 
\right) 
\frac{\Theta_{f,\omega}(s)}{(s-1/2)^k} \, x^{s-\frac{1}{2}} \, ds, 
\endaligned
\end{equation}
where $P_k$ is the polynomial of degree $k-1$ with real coefficients such that 
\[
P_k(\log x) 
= \underset{s=1/2}{\rm Res}\left(\frac{\Theta_{f,\omega}(s)}{(s-1/2)^k} 
x^{s-\frac{1}{2}} \right).
\]
The leading term of $P_k(\log x)$ is 
\[
\Theta_{f,\omega}(1/2) (\log x)^{k-1} = \varepsilon(f) (\log x)^{k-1}. 
\]
By Lemma \ref{lem_502} the fourth and fifth integrals in the right-hand side of \eqref{temp_1} are estimated as 
\begin{equation} \label{temp_2}
\left(
\int_{1/2 +iT}^{c +iT}
- \int_{1/2 -iT}^{c -iT} \right) \frac{\Theta_{f,\omega}(s)}{(s-1/2)^k} \, x^{s-\frac{1}{2}} \, ds
\ll T^{-k} \int_{1/2}^{c} x^{\sigma - \frac{1}{2}} \, d\sigma 
\ll \frac{x^{c-1/2}}{T^k \log x}.  
\end{equation}
By Lemma \ref{lem_501}, \eqref{temp_1}, and \eqref{temp_2}, we obtain
\begin{equation} \label{temp_6}
\aligned
h_{f,\omega}^{\langle k \rangle}(x) 
& = P_k(\log x) + \left( 
-\int_{{\mathcal C}_\delta}
+ \int_{1/2+\delta}^{1/2+i\infty}
+ \int_{1/2-i\infty}^{1/2-i\delta}
\right) 
\frac{\Theta_{f,\omega}(s)}{(s-1/2)^k} \, x^{s-\frac{1}{2}} \, ds \\
& =: P_k(\log x) - I_1 + I_2 + I_3 
\endaligned
\end{equation}
say, by tending $T \to \infty$ for fixed $x \geq 1$. 
Here the integrals $I_2$ and $I_3$ are absolutely integrable, 
since $|\Theta_{f,\omega}(s)|=1$ on the line $\Re(s)=1/2$ by Lemma \ref{lem_502} and $k \geq 2$. 
Therefore, 
\[
I_2 + I_3 =o(1) \quad (x \to \infty)
\]
as a function of $x$ by the Riemann-Lebesgue lemma (\!\!\cite[Theorem 1]{MR942661}). 
In addition, we have  
\[
I_1  \ll \int_{0}^{\pi/2}
x^{-\delta \cos \theta} \, d\theta \ll \frac{1}{\log x} \quad (x \to \infty). 
\]
Hence we obtain 
\[
h_{f,\omega}^{\langle k \rangle}(x) = \varepsilon(f)(\log x)^{k-1} \Bigl(1+O((\log x)^{-1})\Bigr).
\]
In particular $h_{f,\omega}^{\langle k \rangle}$ does not change sign for large $x>0$. \hfill $\Box$
\bigskip

\noindent
{\bf Proof of Theorem \ref{thm_3} (2-b).} 
It is sufficient to prove the case $k=1$ only, 
since the other cases are proved by a similar way with the above. 

Let $0<\omega<1/2$, $0<\delta<\omega$, and $T \geq  {\rm max}\{T_0,  T_1\}$. 
We consider the positively
oriented rectangle with vertices at 
$c + iT$, $1/2 -\delta + iT$, $1/2 -\delta - iT$ and $c - iT$. 
In this rectangle $\Theta_{f,\omega}(s)/(s-1/2)$ has 
no poles except for the simple pole at $s=1/2$ with residue $\varepsilon(f)$ 
by a similar reason with the above proof. 
Thus the residue theorem gives
\begin{equation} \label{temp_3}
\aligned
\frac{1}{2\pi i} & \int_{c-iT}^{c+iT} \frac{\Theta_{f,\omega}(s)}{s-1/2} \,
x^{s-\frac{1}{2}} \, ds   \\
& = \varepsilon(f) + \frac{1}{2\pi i} 
\left( 
 \int_{1/2-\delta-iT}^{1/2-\delta+iT}
+ \int_{1/2-\delta+iT}^{c+iT}
- \int_{1/2-\delta-iT}^{c-iT} 
\right) 
\frac{\Theta_{f,\omega}(s)}{s-1/2} \, x^{s-\frac{1}{2}} \, ds \\
& = \varepsilon(f) + I_1 + I_2 - I_3,
\endaligned
\end{equation}
say. Recall that $d=1$. By Lemma \ref{lem_503} and the Stirling formula \eqref{301}, we have 
\begin{equation} \label{temp_4}
I_1 \ll x^{-\delta} \int_{-T}^{T} \frac{|t|^{\delta+\epsilon}}{1+|t|} \, dt 
\ll (x/T)^{-\delta}T^\epsilon,  
\end{equation}
and 
\begin{equation} \label{temp_5}
\aligned
I_2-I_3 
& \ll T^{\epsilon} \int_{1/2}^{1/2+\omega}  \frac{(x/T)^{\sigma - \frac{1}{2}}}{|\sigma-1/2+iT|} \, d\sigma 
+ T^{-\omega+\epsilon)}\int_{1/2+\omega}^{c}  \frac{x^{\sigma-\frac{1}{2}}}{|\sigma-1/2+iT|} \, d\sigma \\
& \ll \frac{(x/T)^{\omega} - 1}{T^{1-\epsilon}\log(x/T)} 
+ \frac{x^{c-\frac{1}{2}}}{T^{1+\omega-\epsilon}\log x}. 
\endaligned
\end{equation}
By Lemma \ref{lem_501}, \eqref{temp_3}, \eqref{temp_4}, and \eqref{temp_5}, we obtain
\[
\aligned
h_{f,\omega}^{\langle 1 \rangle}(x) 
& =  \varepsilon(f) 
+ O\left( (x/T)^{-\delta}T^\epsilon \right)
+ O\left( \frac{(x/T)^{\omega} - 1}{T^{1-\epsilon}\log(x/T)}\right)
+ O\left( \frac{x^{c-\frac{1}{2}}}{T^{1+\omega-\epsilon}\log x}\right)
 \\
& + O\left(\frac{x^{c-\frac{1}{2}}}{T^{1+\omega}(c-1-\omega)^r}\right)  + O\left(\frac{\psi_{f,\omega}(2x)\, x^{1/2}\log x}{T^{1+\omega}} \right)
 + O\left(\frac{\psi_{f,\omega}(x)}{T^{1+\omega}\sqrt{x}}\right).
\endaligned
\]
By the Ramanujan-Petersson conjecture 
we have $\lambda_f(n) \ll_\epsilon n^{\epsilon}$ and $\mu_f(n) \ll_\epsilon n^{\epsilon}$. 
Therefore we can take $c=1+\omega+\epsilon$ and $\psi_{f,\omega}(x) \ll x^\epsilon$. 
Hence we have 
\[
\aligned
h_{f,\omega}^{\langle 1 \rangle}(x) 
& =  \varepsilon(f) 
+ O\left( (x/T)^{-\delta}T^\epsilon \right)
+ O\left( \frac{(x/T)^{\omega} - 1}{T^{1-\epsilon}\log(x/T)}\right)
+ O\left( \frac{x^{\frac{1}{2}+\omega+\epsilon}}{T^{1+\omega-\epsilon}\log x}\right)
 \\
& + O\left(\frac{x^{\frac{1}{2}+\omega+\epsilon}}{T^{1+\omega}}\right)  
+ O\left(\frac{x^{1/2+\epsilon}\log x}{T^{1+\omega}} \right)
 + O\left(\frac{x^\epsilon}{T^{1+\omega}\sqrt{x}}\right).
\endaligned
\]
By taking $T=x^A$ for some $((1/2)+\omega+\epsilon)/(1+\omega-\epsilon) < A < \delta/(\delta+\epsilon)$ 
(roughly, for some $2/3 < A <1$ by $0<\omega<1/2$), 
we obtain
\[
h_{f,\omega}^{\langle 1 \rangle}(x) = \varepsilon(f) + O(x^{-B}) \quad (x \to \infty)  
\]
for some small $B>0$. 
In particular $h_{f,\omega}^{\langle 1 \rangle}$ does not change sign for large $x>0$. \hfill $\Box$
\bigskip

\noindent
{\bf Proof of Theorem \ref{thm_4}.} 
The first half of Theorem \ref{thm_4} is obvious 
by the proof of Theorem \ref{thm_4} (2-a). 
In fact, the integrand of integrals $I_2$ and $I_3$ in \eqref{temp_6} is $L^2$, 
and hence $I_2 + I_3$ belongs to $L^2((1,\infty),x^{-1}dx)$. 
In addition, as already found, $I_1$ in \eqref{temp_6} also belongs to $L^2((1,\infty),x^{-1}dx)$. 

We prove the latter half of Theorem \ref{thm_4}. By \eqref{403}, we have 
\begin{equation*} 
\int_{1}^{\infty} R_{f,\omega}(x) \, x^{\frac{1}{2}-s} \, \frac{dx}{x} = \frac{\Theta_{f,\omega}(s)-\varepsilon(f)}{s-1/2}   
\end{equation*}
for $s \in \C$ with large $\Re(s)>0$. 
By $R_{f,\omega} \in L^2((1,\infty),x^{-1}dx)$ and   
Theorem 10 of \cite[Chap.II]{MR0005923}, 
we find that $(\Theta_{f,\omega}(s)-\varepsilon(f))/(s-1/2)$ 
has no poles in the right-half plane $\Re(s)>1/2$. 
Then, by the argument of the proof of Theorem \ref{thm_3} (1), 
we arrive at the desired conclusion. \hfill $\Box$

%

%
\end{document}